\documentclass[12pt,reqno]{amsart}
\usepackage{file}

\title{A Gray-categorical pasting theorem}
\author{Nicola Di Vittorio}
\address{Centre of Australian Category Theory, Macquarie University, NSW 2109, Australia}
\email{nicola.divittorio@mq.edu.au}
\date{}

\begin{document}
\begin{abstract}
The notion of $\Gray$-category, a semi-strict $3$-category in which the middle four interchange is weakened to an isomorphism, is central in the study of three-dimensional category theory. In this context it is common practice to use $2$-dimensional pasting diagrams to express composites of $2$-cells, however there is no thorough treatment in the literature justifying this procedure. We fill this gap by providing a formal approach to pasting in $\Gray$-categories and by proving that such composites are uniquely defined up to a contractible groupoid of choices.
\end{abstract}

\maketitle

\section{Introduction}
Pasting diagrams are a graphical tool to express compositions in higher dimensional categories. They can be interpreted as vertical compositions of whiskerings, e.g.\ the triangle identities $g\epsilon \ \circ \ \eta g=\id_g$ and $\epsilon f\circ f\eta=\id_f$ for an adjunction $f\dashv g$ with unit $\eta\colon\id_A\Rightarrow gf$ and counit $\epsilon\colon fg\Rightarrow\id_B$ in a $2$-category can be visualized as follows
\[\begin{tikzcd}
	& A & A & A & A & A && A \\
	B & B && A && B & B & A
	\arrow[""{name=0, anchor=center, inner sep=0}, Rightarrow, no head, from=1-2, to=1-3]
	\arrow[""{name=1, anchor=center, inner sep=0}, "g", from=2-1, to=1-2]
	\arrow[""{name=2, anchor=center, inner sep=0}, Rightarrow, no head, from=2-1, to=2-2]
	\arrow[""{name=3, anchor=center, inner sep=0}, "g"', from=2-2, to=1-3]
	\arrow["f"', from=1-2, to=2-2]
	\arrow[""{name=4, anchor=center, inner sep=0}, "g", bend left=45, pos=0.40, from=2-4, to=1-4]
	\arrow[""{name=5, anchor=center, inner sep=0}, Rightarrow, no head, from=1-5, to=1-6]
	\arrow[""{name=6, anchor=center, inner sep=0}, "f"', from=1-5, to=2-6]
	\arrow[""{name=7, anchor=center, inner sep=0}, Rightarrow, no head, from=2-6, to=2-7]
	\arrow[""{name=8, anchor=center, inner sep=0}, "f", from=1-6, to=2-7]
	\arrow[""{name=9, anchor=center, inner sep=0}, "f", bend left=45, pos=0.60, from=1-8, to=2-8]
	\arrow["g", from=2-6, to=1-6]
	\arrow[""{name=10, anchor=center, inner sep=0}, "g"', bend right=45, pos=0.40, from=2-4, to=1-4]
	\arrow[""{name=11, anchor=center, inner sep=0}, "f"', bend right=45, pos=0.60, from=1-8, to=2-8]
	\arrow["\epsilon", shorten <=4pt, shorten >=4pt, Rightarrow, xshift=2mm, from=1, to=2]
	\arrow["\eta"', shorten <=4pt, shorten >=4pt, Rightarrow, xshift=-2mm, from=0, to=3]
	\arrow["\eta", shorten <=4pt, shorten >=4pt, Rightarrow, from=5, to=6]
	\arrow[shorten <=10pt, shorten >=10pt, Rightarrow, no head, from=10, to=4]
	\arrow[shorten <=10pt, shorten >=10pt, Rightarrow, no head, from=11, to=9]
	\arrow[shorten <=38pt, shorten >=38pt, Rightarrow, no head, from=3, to=10]
	\arrow["\epsilon"', shorten <=4pt, shorten >=4pt, Rightarrow, xshift=-2mm, from=8, to=7]
	\arrow[xshift=8mm, shorten <=22pt, shorten >=42pt, Rightarrow, no head, from=8, to=11]
\end{tikzcd}\]
While in the situation above the composite is always uniquely determined, there are cases when it is not clear how to intepret a pasting diagram. For instance 
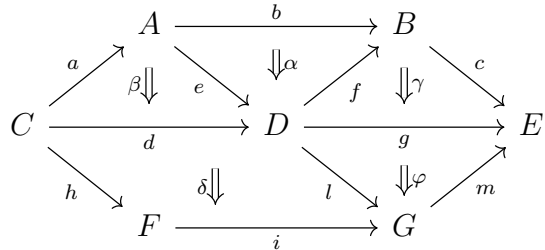
\[\begin{tikzcd}
	& A && B \\
	C && D && E \\
	& F && G
	\arrow["a", from=2-1, to=1-2]
	\arrow[""{name=0, anchor=center, inner sep=0}, "b", from=1-2, to=1-4]
	\arrow["c", from=1-4, to=2-5]
	\arrow["e"', from=1-2, to=2-3]
	\arrow["f"', from=2-3, to=1-4]
	\arrow[""{name=1, anchor=center, inner sep=0}, "g"', from=2-3, to=2-5]
	\arrow[""{name=2, anchor=center, inner sep=0}, "d"', from=2-1, to=2-3]
	\arrow["h"', from=2-1, to=3-2]
	\arrow[""{name=3, anchor=center, inner sep=0}, "i"', from=3-2, to=3-4]
	\arrow["l"', from=2-3, to=3-4]
	\arrow["m"', from=3-4, to=2-5]
	\arrow["\beta"', shorten <=8pt, shorten >=8pt, Rightarrow, from=1-2, to=2]
	\arrow["\alpha", shorten <=8pt, shorten >=8pt, Rightarrow, from=0, to=2-3]
	\arrow["\gamma", shorten <=8pt, shorten >=8pt, Rightarrow, from=1-4, to=1]
	\arrow["\delta"', shift right=8, shorten <=8pt, shorten >=8pt, Rightarrow, from=2-3, to=3]
	\arrow["\varphi", yshift=-2mm, shorten <=8pt, shorten >=8pt, Rightarrow, from=1, to=3-4]
\end{tikzcd}\]
can be written as vertical composition in two different ways, namely
$$ m\delta\cdot\varphi d\cdot\gamma d\cdot cf\beta\cdot c\alpha a$$ and
$$ m\delta\cdot\varphi d\cdot g\beta\cdot\gamma ea\cdot c\alpha a,$$
which coincide for strict higher categories such as $2$-categories by naturality of whiskering, that is in turn an immediate consequence of the middle four interchange law (see Lemma B.1.3 in \cite{RVbook} for a proof), so that the square
\[\begin{tikzcd}
	cfea & cfd \\
	gea & gd
	\arrow["cf\beta", Rightarrow, from=1-1, to=1-2]
	\arrow["{\gamma d}", Rightarrow, from=1-2, to=2-2]
	\arrow["{\gamma ea}"', Rightarrow, from=1-1, to=2-1]
	\arrow["g\beta"', Rightarrow, from=2-1, to=2-2]
\end{tikzcd}\]
commutes. A number of results has been proven in this setting, such as in \cite{johnson1987pasting} and \cite{power19902}. The latter has been extended to bicategories in \cite{verity2011enriched}. A didactic account of these two results can be found in the book \cite{johnson20212}. In short, \cite{power19902} provides a basic algorithm to get a composite for a pasting diagram as a vertical composition of whiskered $2$-cells. In each step of the algorithm we remove a $2$-cell and add it to the whiskered composite, going from top to bottom. So, for instance, in our example we start by taking off $\alpha$ but then we can either remove $\beta$ or $\gamma$. If we choose $\beta$ we will remove $\gamma$ in the next step and viceversa. We just keep removing $2$-cells until none is left. This will theoretically produce different composites, but each time there is a choice between two or more $2$-cells we can use the middle four interchange law so that in the end all the composites will be equal. In the following picture, the columns describe the two compositions of the previous pasting diagram.
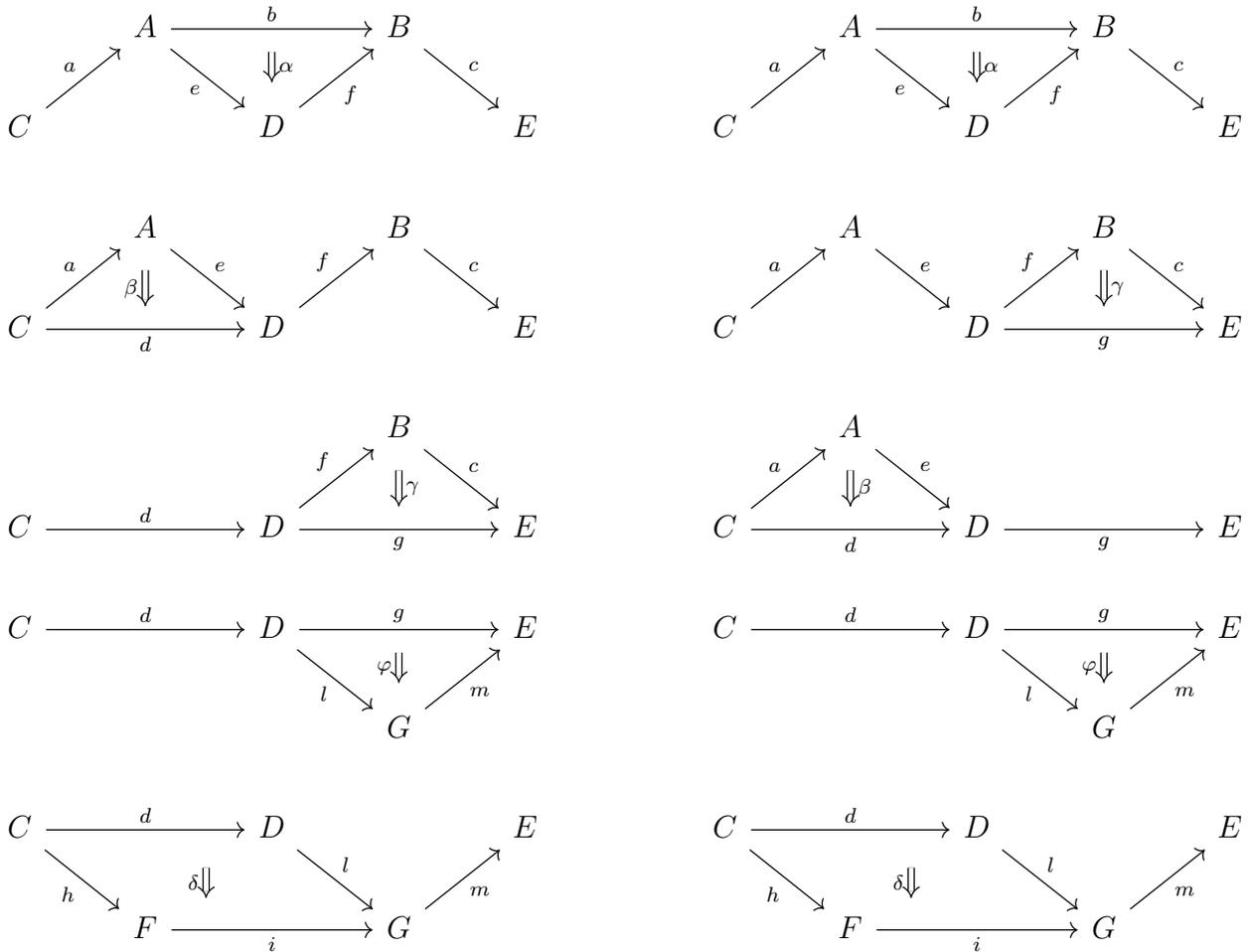
\[\begin{tikzcd}
	& A && B &&&& A && B \\
	C && D && E && C && D && E \\
	& A & {} & B &&& {} &  A &&  B \\
	C && D && E && C && D && E \\
	&&& B &&&& A \\
	C && D && E && C && D && E \\
    C && D && E && C && D && E \\
	&&& G &&&&&& G \\
	C && D && E && C && D && E \\
	& F && G &&&&  F && G
	\arrow[""{name=0, anchor=center, inner sep=0}, "b", from=1-2, to=1-4]
	\arrow["e"', from=1-2, to=2-3]
	\arrow["f"', from=2-3, to=1-4]
	\arrow["a", from=2-1, to=1-2]
	\arrow["a", from=4-1, to=3-2]
	\arrow["e", from=3-2, to=4-3]
	\arrow[""{name=1, anchor=center, inner sep=0}, "d"', from=4-1, to=4-3]
	\arrow["f", from=4-3, to=3-4]
	\arrow["c", from=3-4, to=4-5]
	\arrow["l"', from=7-3, to=8-4]
	\arrow[""{name=2, anchor=center, inner sep=0}, "g", from=7-3, to=7-5]
	\arrow["m"', from=8-4, to=7-5]
	\arrow["d", from=7-1, to=7-3]
	\arrow["m"', from=10-4, to=9-5]
	\arrow["i"', from=10-2, to=10-4]
	\arrow["h"', from=9-1, to=10-2]
	\arrow[""{name=3, anchor=center, inner sep=0}, "d", from=9-1, to=9-3]
	\arrow["l", from=9-3, to=10-4]
	\arrow["a", from=4-7, to=3-8]
	\arrow["e", from=3-8, to=4-9]
	\arrow["f", from=4-9, to=3-10]
	\arrow["c", from=3-10, to=4-11]
	\arrow[""{name=4, anchor=center, inner sep=0}, "g"', from=4-9, to=4-11]
	\arrow["a", from=6-7, to=5-8]
	\arrow["d", from=6-1, to=6-3]
	\arrow["f", from=6-3, to=5-4]
	\arrow["c", from=5-4, to=6-5]
	\arrow[""{name=5, anchor=center, inner sep=0}, "g"', from=6-3, to=6-5]
	\arrow["e", from=5-8, to=6-9]
	\arrow[""{name=6, anchor=center, inner sep=0}, "d"', from=6-7, to=6-9]
	\arrow["g"', from=6-9, to=6-11]
	\arrow["d", from=7-7, to=7-9]
	\arrow[""{name=7, anchor=center, inner sep=0}, "g", from=7-9, to=7-11]
	\arrow["l"', from=7-9, to=8-10]
	\arrow["m"', from=8-10, to=7-11]
	\arrow["c", from=1-4, to=2-5]
	\arrow["a", from=2-7, to=1-8]
	\arrow[""{name=8, anchor=center, inner sep=0}, "b", from=1-8, to=1-10]
	\arrow["e"', from=1-8, to=2-9]
	\arrow["f"', from=2-9, to=1-10]
	\arrow["c", from=1-10, to=2-11]
	\arrow[""{name=9, anchor=center, inner sep=0}, "d", from=9-7, to=9-9]
	\arrow["l", from=9-9, to=10-10]
	\arrow["h"', from=9-7, to=10-8]
	\arrow["i"', from=10-8, to=10-10]
	\arrow["m"', from=10-10, to=9-11]
	\arrow["\alpha", shorten <=8pt, shorten >=8pt, Rightarrow, from=0, to=2-3]
	\arrow["\gamma", shorten <=8pt, shorten >=8pt, Rightarrow, from=3-10, to=4]
	\arrow["\varphi"', shorten <=8pt, shorten >=8pt, Rightarrow, from=2, to=8-4]
	\arrow["\beta"', shorten <=8pt, shorten >=8pt, Rightarrow, from=3-2, to=1]
	\arrow["\delta"', shift left=8, yshift=-2mm, shorten <=8pt, shorten >=8pt, Rightarrow, from=3, to=10-2]
	\arrow["\delta"', shift left=8, yshift=-2mm, shorten <=8pt, shorten >=8pt, Rightarrow, from=9, to=10-8]
	\arrow["\gamma", shorten <=8pt, shorten >=8pt, Rightarrow, from=5-4, to=5]
	\arrow["\beta", shorten <=8pt, shorten >=8pt, Rightarrow, from=5-8, to=6]
	\arrow["\varphi"', shorten <=8pt, shorten >=8pt, Rightarrow, from=7, to=8-10]
	\arrow["\alpha", shorten <=8pt, shorten >=8pt, Rightarrow, from=8, to=2-9]
\end{tikzcd}\]
However for more general weak higher categories the interchange law might hold just up to coherent isomorphism so this algorithm needs to be modified. Therefore the uniqueness of the pasting composite must be interpreted in a suitable way, namely as a contractibility condition on the space of composites of the pasting diagram (see for instance \cite{hackney2021infty}). This is indeed the case for the current work, where we will be dealing with the semi-strict case of a pasting diagram in a $\Gray$-category, a particular notion of $3$-dimensional category where the middle four interchange law is not strict but it is instead part of coherence data. In particular, we will prove the following theorem.
\begin{mythm}{4.24}
Every $2$-dimensional pasting diagram in a $\Gray$-category has a unique composition up to a contractible groupoid of choices.
\end{mythm}
It has to be noticed that we are only considering pasting composites of $2$-cells. Furthermore the composites with which we are dealing are the ones obtained as outputs of the (nondeterministic) algorithm described earlier in the Introduction. With some effort one could show that every possible composite $-$ living in the free $\Gray$-category on the $\Gray$-computad underlying the $2$-dimensional pasting diagram $-$ arises in this way, but it would be beyond the scope of this paper which is actually concerned with giving a justification to the practice of pasting inside a $\Gray$-category for how is routinely done. As far as we know relating these two notions, one more geometrical and the other more combinatorial in nature, is still an open problem. The result we present makes precise an observation that can be found in section $5.2$ of the seminal work \cite{gordon1995coherence}. It also provides a proof to a conjecture stated in Remark $2.2.14$ of \cite{NDVmres}.

The proof of our pasting theorem uses techniques from rewriting theory, but no previous knowledge of it is required. We will recall the basics of this theory along the way.

\textbf{Acknowledgments.} I am deeply grateful to my supervisor Dominic Verity for the many insights that he shared with me while I was writing this paper and to the anonymous referee for greatly improving the readability of the text, as well as pointing out some mistakes that have now been corrected. I also acknowledge the support of an International Macquarie University Research Excellence Scholarship.

\section{Preliminaries on relations}
At first let us recall some basic facts about relations, since we will use them later for rewriting. In accordance with the literature, whenever $\cR$ a relation on a set $X$ we will write $x\cancel\cR y$ to denote that $(x,y)\notin \cR$.
\begin{defn}
A relation $\cR$ on a set $X$ is said to be \emph{irreflexive} if $\forall x\in X \ x\cancel\cR x$.
\end{defn}

\begin{defn}
A relation $\cR$ on a set $X$ is said to be \emph{asymmetric} if for all $x,y\in X$ we have that $x\cR y \implies y\cancel\cR x$. An irreflexive, asymmetric and transitive relation is called a \emph{strict partial order}.
\end{defn}
\begin{rmk}\label{irr and tr}
An irreflexive and transitive relation is also asymmetric, hence a strict partial order. 
\end{rmk}
A set equipped with a strict partial order will be called \emph{strict poset}. A \emph{strict linear order} (also called \emph{strict total order}) is a strict partial order for which any two elements are comparable.  \begin{defn} 
We say that $(A,<^*)$ is a \emph{strict linear extension} of a strict poset $(A,<)$ if 
\begin{enumerate}
    \item $<^*$ is a \emph{strict linear order};
    \item for every $a,b\in A$, we have that $a < b\implies a<^* b$.
\end{enumerate}
In other words, a strict linear extension is a strict linear order that contains the given partial order.
\end{defn}
\begin{rmk}
In general there is more than one linear extension of a given poset, for instance $(\{x, y\}, =)$ with $x\ne y$ can be extended to a linear order either by choosing $x< y$ or $y < x$. 
\end{rmk}
Given an irreflexive relation $\cR$ we can ask if its transitive closure is still irreflexive, so that by Remark \ref{irr and tr} it is a strict partial order. For this to happen it is enough that $\cR$ is \emph{acyclic}, i.e.\ there are no $x_1, x_2, \dots, x_n\in X$ s.t.\ $x_1\cR x_2$ and $x_2\cR x_3$ and \dots and $x_n\cR x_1$ ($x_1\cR x_2\cR x_3\cdots x_n\cR x_1$ for short). Whenever this condition holds, we define a strict linear extension of $\cR$ to be a strict linear extension of its transitive closure. The following proposition guarantees that in such a case a linear extension always exists.  
\begin{prop}\label{acyclic iff well founded}
Let $(X, \cR)$ be a finite set endowed with an irreflexive relation. The following are equivalent:
\begin{enumerate}
    \item[(a)] $\cR$ is acyclic;
    \item[(b)] $\cR$ is well founded, i.e.\ $\forall S\subseteq X, \  S\ne\emptyset, \ \exists m\in S \ \forall s\in S \ s\cancel\cR m$ $($called a \emph{minimal element}$)$;
    \item[(c)] $\cR$ admits a strict linear extension $<$.
    \end{enumerate}
\end{prop}
\begin{proof}
$(a)\Rightarrow(b)$ Since $S\ne\emptyset$, there exists $m_1\in S$. If $\forall s\in S \ s\cancel\cR m_1$, we have $m=m_1$, otherwise there exists $m_2\in S$ s.t.\ $m_2\cR m_1$. For the same reason, either $m=m_2$ or there exists $m_3\in S$ with $m_3\cR m_2$. In the latter case, iterating this argument eventually gives an element $m_i$ we already visited (since $S$ is finite) and therefore a cycle $m_i\cR m_k\cdots m_{i+1}\cR m_i$. This contradicts the assumption that $\cR$ is acyclic, and so one of the elements we visited before $m_i$ must be minimal.    

$(b)\Rightarrow(c)$ The whole set $X$ is a subset of itself, so it has a minimal element $x_1$. The set $X\setminus\{x_1\}$ is contained in $X$, therefore has a minimal element $x_2$. We put $x_1<x_2$ in the linear extension. This choice is allowed since $x_2$ is not related to $x_1$ in the transitive closure of $\cR$ for the minimality of $x_1$ in $X$. We can go on with this procedure and build a descending chain
$$X_1=X \supset X_2=X\setminus\{x_1\}\supset X_3=X\setminus \{x_1,x_2\}\supset\cdots\supset X_{n+1}=\emptyset$$
of finite length since $X$ is finite, corresponding to the linear order on $X=\{x_1,\dots,x_n\}$ given by
$$x_1<x_2<\cdots<x_n$$
which is compatible with $\cR$ by construction.

$(c)\Rightarrow (a)$ A cycle $y_1\cR y_2\cR\cdots y_l\cR y_1$ cannot be ordered: a linear extension of $\cR$ would have to satisfy $y_1<y_2<\cdots<y_n<y_1$ and then, by transitivity of $<$, we have $y_1< y_1$ contradicting the irreflexivity of $<$. Therefore if $\cR$ has a cycle, it cannot be extended to a strict linear order. \end{proof}

\section{Rewriting Systems}
Rewriting theory is the main tool we will be using to prove the pasting theorem for $\Gray$-categories. It is indeed useful to think about the groupoid appearing in the claim of the theorem in terms of generators and relations, so that it can be studied using rewriting. For this reason, here we will briefly introduce the fundamental notion of rewriting system and the most important results related to it. For a thorough exposition of rewriting theory we refer to  \cite{baader1999term} and \cite{barendsen2003term}.
\begin{defn}
A \emph{rewriting system} is a set $A$ equipped with a binary relation $\rightarrow$, called \emph{reduction}.
\end{defn}
The idea is that if $a\rightarrow b$, we can substitute any occurrence of $a$ with $b$. For instance, in the theory of groups we can consider the free group on a set and represent its elements as words made of generators, some of whom may be formally inverted. Then we have a reduction $gg^{-1}\rightarrow e$ so we can replace every consecutive product of an element by its inverse with the identity element. 

We will denote by $\xrightarrow{*}$ the reflexive transitive closure of $\rightarrow$, namely the smallest preorder containing $\rightarrow$.

\begin{defn}
 An element $a\in A$ is said to be \emph{confluent} if for all $b,c\in A$ s.t.\ $a\xrightarrow{*} b$ and $a\xrightarrow{*} c$ there exists $d\in A$ s.t.\ $b\xrightarrow{*} d$ and $c\xrightarrow{*} d$. A rewriting system is called confluent if all its elements are confluent. 
\end{defn}

\begin{defn}
 An element $a\in A$ is said to be \emph{locally confluent} if for all $b,c\in A$ s.t.\ $a\rightarrow b$ and $a\rightarrow c$ there exists $d\in A$ s.t.\ $b\xrightarrow{*} d$ and $c\xrightarrow{*} d$. A rewriting system is called locally confluent if all its elements are locally confluent. 
\end{defn}

\begin{rmk}
The difference between local confluence and confluence is that in the former we have a one-step reduction from $a$ to $b$ and $c$, while in the latter we can reach $b$ and $c$ in more than one step.
\end{rmk}

\begin{defn}
A rewriting system is called \emph{terminating} if there is no infinite chain of the form $$a_0\rightarrow a_1 \rightarrow a_2 \rightarrow \cdots$$
\end{defn}
\noindent A trick that is often useful in showing that a rewriting system is terminating is to define a measure $\rho\colon X\to\N$ that is reduced by any application of the rewrite $\rightarrow$, namely $x\rightarrow y\Rightarrow\rho(x)>\rho(y)$.
\begin{lemma}[Newman's Lemma]
A terminating rewriting system is confluent if and only if it is locally confluent. 
\end{lemma}
A short proof of this lemma, using induction, can be found in \cite{huet1980confluent}. For a given $a\in A$ and rewrites $a\xrightarrow{*} b$ and $a\xrightarrow{*} c$, a key point in proving that a diamond
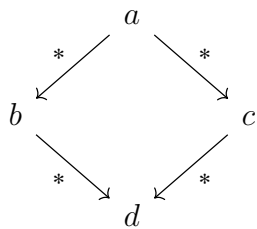
\[\begin{tikzcd}
	& a \\
	b && c \\
	& d
	\arrow["{*}"', from=1-2, to=2-1]
	\arrow["{*}"', from=2-1, to=3-2]
	\arrow["{*}", from=1-2, to=2-3]
	\arrow["{*}", from=2-3, to=3-2]
\end{tikzcd}\]
does indeed exist is to use induction on the derivation length and tessellate it as follows
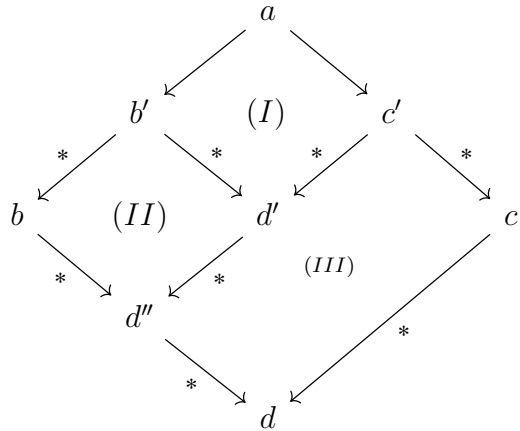
\[\begin{tikzcd}
	&& a \\
	& {b'} && {c'} \\
	b && {d'} && c \\
	& {d''} \\
	&& d
	\arrow[from=1-3, to=2-2]
	\arrow["{*}"', from=2-2, to=3-1]
	\arrow["{*}"', from=3-1, to=4-2]
	\arrow["{*}", from=2-2, to=3-3]
	\arrow["{*}", from=3-3, to=4-2]
	\arrow[from=1-3, to=2-4]
	\arrow["{*}"', from=2-4, to=3-3]
	\arrow["{*}", from=2-4, to=3-5]
	\arrow["{*}", from=3-5, to=5-3]
	\arrow["{*}"', from=4-2, to=5-3]
	\arrow["{(I)}"{marking}, draw=none, from=2-2, to=2-4]
	\arrow["{(II)}"{marking}, draw=none, from=3-1, to=3-3]
	\arrow["{(III)}"{description}, draw=none, from=4-2, to=3-5]
\end{tikzcd}\]
obtaining $(I)$ by local confluence, while $(II)$ and $(III)$ follow by inductive hypothesis. This idea will also be important in the proof of the pasting theorem, where the smaller diagrams are actually commutative.

An important consequence of Newman's lemma is the existence and unicity of a minimal element (called \emph{normal form}) in every connected component\footnote{We define a connected component of a relation $(X,\cR)$ as a connected component of the corresponding directed graph having $V=X$ and a directed edge $x\to y$ whenever $x\cR y$.} of $\rightarrow$.

\section{The pasting theorem}
In this section we will provide the proof of the main result, namely Theorem \ref{pasting theorem}, using rewriting techniques. Before that, we review the notion of $\Gray$-category and the graph-theoretical concepts needed to formalize the intuition behind pasting diagrams.

A $\Gray$-category is a particular instance of enriched category (see \cite{kelly1982basic} for the general definition of enriched category). In particular, we can define it in a very concise way as follows.
\begin{defn}
A $\Gray$-category is a category enriched over the monoidal category $(\twoCat, \otimes, \bb1)$ of $2$-categories and strict $2$-functors equipped with the Gray tensor product.
\end{defn}
Unpacking this definition, a $\Gray$-category $\cK$ consists of the following data:
\begin{enumerate}[i)]
    \item a class of objects $\cK_0$,
    \item for each couple of objects $A$ and $B$ in $\cK$, a $2$-category $\cK(A,B)$,
    \item for every $A\in\cK$ an identity $1$-cell $\id_A\colon A\to A$,
    \item a composition $2$-functor $c_{A,B,C}\colon\cK(B,C)\otimes\cK(A,B)\to\cK(A,C)$ from the Gray tensor product between the $2$-categories $\cK(B,C)$ and $\cK(A,B)$ satisfying associativity and unitality rules.  
\end{enumerate}
Explicitly, the $2$-category $\cK(B,C)\otimes\cK(A,B)$ is defined as follows:
\begin{itemize}
    \item $\Ob(\cK(B,C)\otimes\cK(A,B))=\Ob(\cK(B,C))\times\Ob(\cK(A,B))$,
    \item $1$-cells generated by $(\alpha,g)\colon (f,g)\to (f',g)$ and $(f,\beta)\colon (f,g)\to(f,g')$ with $\alpha\colon f\to f'$ in $\cK(B,C)$ and $\beta\colon g\to g'$ in $\cK(A,B)$ subject to the relations $(\alpha'\alpha,g)=(\alpha',g)(\alpha,g)$, $(f,\beta'\beta)=(f,\beta')(f,\beta)$ whenever these pairs are composable and $\id_{(f,g)}=(\id_f,g)=(f,\id_g)$. 
    \item $2$-cells generated by 
    \[\begin{tikzcd}
	{(f,g)} && {(f',g)} & {(f,g)} && {(f,g')}
	\arrow[""{name=0, anchor=center, inner sep=0}, "{(\alpha,g)}", bend left=45, from=1-1, to=1-3]
	\arrow[""{name=1, anchor=center, inner sep=0}, "{(\alpha',g)}"', bend right=45, from=1-1, to=1-3]
	\arrow[""{name=2, anchor=center, inner sep=0}, "{(f,\beta)}", bend left=45, from=1-4, to=1-6]
	\arrow[""{name=3, anchor=center, inner sep=0}, "{(f,\beta')}"', bend right=45, from=1-4, to=1-6]
	\arrow["{(\Phi,g)}", shorten <=17pt, shorten >=17pt, Rightarrow, from=0, to=1]
	\arrow["{(f,\Psi)}", shorten <=17pt, shorten >=17pt, Rightarrow, from=2, to=3]
\end{tikzcd}\]
for any $\Phi\colon\alpha\Rightarrow\alpha'$ in $\cK(B,C)$ and $\Psi\colon\beta\Rightarrow\beta'$ in $\cK(A,B)$ satisfying relations for vertical and horizontal compositions similar to the ones we have for $1$-cells (see Definition 12.2.5 in \cite{johnson20212} for more details). In addition, we have generating $2$-cells (sometimes called \emph{Gray cells}):
\[\begin{tikzcd}
	{(f,g)} & {(f,g')} \\
	{(f',g)} & {(f',g')}
	\arrow["{(\alpha,g')}", from=1-2, to=2-2]
	\arrow["{(\alpha,g)}"', from=1-1, to=2-1]
	\arrow["{(f',\beta)}"', from=2-1, to=2-2]
	\arrow["{(f,\beta)}", from=1-1, to=1-2]
	\arrow["{\gamma_{\alpha,\beta}}", shorten <=13pt, shorten >=13pt, Rightarrow, from=1-2, to=2-1]
\end{tikzcd}\]
which are invertible for the pseudo version and oriented in either way for the lax/colax version of the Gray tensor product. These $2$-cells are subject to the relations 

\[\begin{tikzcd}
	{(f,g)} & {(f,g')} & {(f,g)} & {(f,g')} \\
	{(f',g)} & {(f',g')} & {(f',g)} & {(f',g')}
	\arrow["{(f,\beta)}", from=1-1, to=1-2]
	\arrow[""{name=0, anchor=center, inner sep=0}, "{(\alpha,g')}", from=1-2, to=2-2]
	\arrow[""{name=1, anchor=center, inner sep=0}, "{(\alpha,g)}", from=1-1, to=2-1]
	\arrow["{(f',\beta)}"', from=2-1, to=2-2]
	\arrow[""{name=2, anchor=center, inner sep=0}, "{(\alpha',g)}"', bend right = 75pt, from=1-1, to=2-1]
	\arrow[""{name=3, anchor=center, inner sep=0}, "{(\alpha',g)}"', from=1-3, to=2-3]
	\arrow["{(f,\beta)}", from=1-3, to=1-4]
	\arrow["{(f',\beta)}"', from=2-3, to=2-4]
	\arrow[""{name=4, anchor=center, inner sep=0}, "{(\alpha',g')}"', from=1-4, to=2-4]
	\arrow[""{name=5, anchor=center, inner sep=0}, "{(\alpha,g')}", bend left = 75pt, from=1-4, to=2-4]
	\arrow["{\gamma_{\alpha',\beta}}"', shorten <=13pt, shorten >=13pt, Rightarrow, from=1-4, to=2-3]
	\arrow["{\gamma_{\alpha,\beta}}", shorten <=13pt, shorten >=13pt, Rightarrow, from=1-2, to=2-1]
	\arrow["{=}"{description}, draw=none, from=0, to=3]
	\arrow["{(\Phi,g')}", shorten <=7pt, shorten >=7pt, Rightarrow, from=5, to=4]
	\arrow["{(\Phi,g)}"', shorten <=7pt, shorten >=7pt, Rightarrow, from=1, to=2]
\end{tikzcd}\]
\[\begin{tikzcd}
	{(f,g)} & {(f,g')} & {(f,g)} & {(f,g')} \\
	{(f,g)} & {(f,g')} & {(f,g)} & {(f,g')}
	\arrow["{(f,\beta)}", from=1-1, to=1-2]
	\arrow[""{name=0, anchor=center, inner sep=0}, "{(\id_f,g')}", from=1-2, to=2-2]
	\arrow["{(\id_f,g)}"', from=1-1, to=2-1]
	\arrow["{(f,\beta)}"', from=2-1, to=2-2]
	\arrow["{(f,\beta)}", from=1-3, to=1-4]
	\arrow["{\id_{(f,g')}}", from=1-4, to=2-4]
	\arrow[""{name=1, anchor=center, inner sep=0}, "{\id_{(f,g)}}"', xshift=1mm, from=1-3, to=2-3]
	\arrow["{(f,\beta)}"', from=2-3, to=2-4]
	\arrow["{\gamma_{\id_f,\beta}}"', shorten <=12pt, shorten >=12pt, Rightarrow, from=1-2, to=2-1]
	\arrow["{=}"{marking}, draw=none, from=2-3, to=1-4]
	\arrow["{=}"{description}, xshift=1mm, Rightarrow, draw=none, from=0, to=1]
\end{tikzcd}\]
\newline
\[\begin{tikzcd}
	{(f,g)} & {(f,g')} & {(f,g)} & {(f,g')} \\
	{(f',g)} & {(f',g')} \\
	{(f'',g)} & {(f'',g')} & {(f'',g)} & {(f'',g')}
	\arrow["{(f,\beta)}", from=1-1, to=1-2]
	\arrow["{(\alpha,g')}", from=1-2, to=2-2]
	\arrow["{(\alpha,g)}"', from=1-1, to=2-1]
	\arrow["{(f',\beta)}", from=2-1, to=2-2]
	\arrow["{(\alpha',g)}"', from=2-1, to=3-1]
	\arrow["{(f'',\beta)}"', from=3-1, to=3-2]
	\arrow["{(\alpha',g')}", from=2-2, to=3-2]
	\arrow["{\gamma_{\alpha,\beta}}"', shorten <=13pt, shorten >=13pt, Rightarrow, from=1-2, to=2-1]
	\arrow["{\gamma_{\alpha',\beta}}"', shorten <=13pt, shorten >=13pt, Rightarrow, from=2-2, to=3-1]
	\arrow["{(f,\beta)}", from=1-3, to=1-4]
	\arrow[""{name=0, anchor=center, inner sep=0}, "{(\alpha'\alpha,g)}"{description}, from=1-3, to=3-3]
	\arrow["{(f'',\beta)}"', from=3-3, to=3-4]
	\arrow["{(\alpha'\alpha,g')}"{description}, from=1-4, to=3-4]
	\arrow["{\gamma_{\alpha'\alpha,\beta}}"'{pos=0.4}, shorten <=34pt, shorten >=34pt, Rightarrow, from=1-4, to=3-3]
	\arrow["{=}"{marking, pos=0.3}, draw=none, from=2-2, to=0]
\end{tikzcd}\]
as well as their horizontal analogues.
\end{itemize}
In this paper we deal only with $\Gray$-categories in their pseudo version. Namely, categories enriched over $\twoCat$ equipped with the pseudo-Gray tensor product or equivalently $\Gray$-categories whose Gray cells are invertible. We turn now to make precise the notion of pasting diagram inside a $\Gray$-category. The key idea is to capture the structure of our pasting diagrams using a graph which will be labelled in components of a $\Gray$-category. For us a graph will be given by a pair of sets $(V,E)$, respectively called vertices and edges, with $E$ being a set of paired vertices.
\begin{defn}
A graph is said to be \emph{finite} if both the vertex and edge sets are finite.
\end{defn}
\begin{defn}
A directed graph is \emph{connected} if its underlying undirected graph is connected.  
\end{defn}
\begin{defn}
 A \emph{plane graph} is a graph together with a specified embedding of it in the plane.  
\end{defn}
We assume throughout the paper that the plane is oriented with the usual orientation and that the aforementioned embedding preserves this orientation. 
All of the following graphs will be assumed to be plane, directed, connected and finite. 
\begin{defn}
A \emph{graph with source $s$ and sink $t$} is a graph with distinct vertices $s$ and $t$ such that for every vertex $v$ there exist directed paths from $s$ to $v$ and from $v$ to $t$.
\end{defn}
The previous definition makes sense also for graphs that are not plane, but in the following we will focus on the plane case. We can now introduce the fundamental graph-theoretical tool used to formalize pasting diagrams, as defined in \cite{power19902}, that we will employ in the rest of the paper.
\begin{defn}
A \emph{pasting scheme} $\mathcal{G}$ is an \emph{acyclic} graph with source $s$ and sink $t$.
\end{defn}
\begin{figure}
    \centering
   \[\begin{tikzcd}
	&&& {v_1} & {v_2} \\
	{} & {} & s & {v_3} & {} & t \\
	&&& {v_4} & {v_5}
	\arrow[from=2-3, to=2-4]
	\arrow[from=2-3, to=1-4]
	\arrow[from=1-4, to=1-5]
	\arrow[from=2-4, to=1-5]
	\arrow[from=1-5, to=2-6]
	\arrow[from=2-4, to=2-6]
	\arrow[from=2-3, to=3-4]
	\arrow[from=3-4, to=3-5]
	\arrow[from=3-5, to=2-6]
	\arrow[from=2-4, to=3-5]
	\arrow["{F_1}"{description}, draw=none, from=1-4, to=2-4]
	\arrow["{F_2}"{description}, draw=none, from=1-5, to=2-5]
	\arrow["{F_3}"{description}, draw=none, from=2-4, to=3-4]
	\arrow["{F_4}"{description}, draw=none, from=2-5, to=3-5]
\end{tikzcd}\]
    \caption{Example of pasting scheme}
    \label{fig:pasting scheme}
\end{figure}
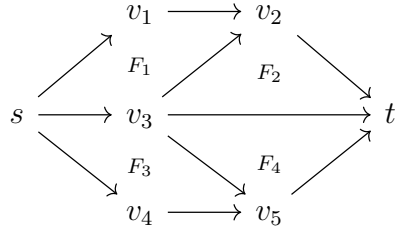
It is important to notice that the acyclicity condition for pasting schemes refers to the lack of \emph{directed} cycles. The underlying undirected graph of a pasting scheme has generally many cycles, namely the boundaries of the internal faces of the pasting scheme. In \cite{power19902} the author also proves the following characterization of pasting schemes.
\begin{prop} A graph $\mathcal{G}$ with source $s$ and sink $t$ is a pasting scheme if and only if for every interior face $F$ there exist distinct vertices $s_F$ and $t_F$ and directed paths $\sigma_F$ and $\tau_F$ from $s_F$ to $t_F$ such that the boundary of $F$ is the directed path $\sigma_F\tau_F^*$, where by $\tau_F^*$ we mean that $\tau_F$ is traversed in the opposite way. The vertices $s_F, t_F$ and directed paths $\sigma_F, \tau_F$ are necessarily unique.
\end{prop}
If $E$ is the exterior face of $\mathcal{G}$, then $\tau_E$ is the \emph{top} path from $s$ to $t$ and $\sigma_E$ is the \emph{bottom} path.
\begin{defn} For a couple of faces $F_1$ and $F_2$ in a pasting scheme $\mathcal{G}$, we define a relation $F_1\triangleleft_{\mathcal{G}} F_2$ if and only if $\tau_{F_1}$ and $\sigma_{F_2}$ share at least one edge. 
\end{defn}

\begin{rmk}\label{properties triangl}
The relation $\triangleleft_{\mathcal{G}}$ has the following properties:
\begin{itemize}
    \item it is irreflexive, by Corollary $2.7 \ (2)$ of \cite{power19902}, 
    \item it can be extended to a strict linear order, concretely provided by the algorithm used in Proposition $2.10$ of \cite{power19902} to find a minimal element for $\triangleleft_{\mathcal{G}}$.
\end{itemize}
Hence, by Proposition \ref{acyclic iff well founded}, it is also acyclic and well-founded. From now on we will implicitly assume that every path is directed.
\end{rmk}
\begin{defn}
Given two faces $F_1$ and $F_2$ in a pasting scheme $\mathcal{G}$, we define a relation $F_1\prec_{\mathcal{G}} F_2$ if and only if there exists a (possibly empty) path $t_{F_1}\rightsquigarrow s_{F_2}$.
\end{defn}
\begin{prop}
The relation $F_1\prec_{\mathcal{G}} F_2$ is a strict partial order.  \end{prop}
\begin{proof} We write $t_{F_1}\xra{F_1\prec_{\mathcal{G}} F_2}s_{F_2}$ to represent a path from $t_{F_1}$ to $s_{F_2}$ that witnesses their relationship. Let us prove that the relation $\prec_{\mathcal{G}}$ satisfies irreflexivity, transitivity and asymmetry. 

\noindent \emph{Irreflexivity}: if not, we would have the directed cycle 
$$t_{F_1}\xra{F_1\prec_{\mathcal{G}} F_1} s_{F_1}\xra{\sigma_{F_1}} t_{F_1}$$
\emph{Transitivity}: suppose $F_1\prec_{\mathcal{G}} F_2$ and $F_2\prec_{\mathcal{G}} F_3$ therefore there exist $t_{F_1}\rightsquigarrow s_{F_2}$ and
$t_{F_2}\rightsquigarrow s_{F_3}$, giving a directed path
$$t_{F_1}\xra{F_1\prec_{\mathcal{G}} F_2} s_{F_2}\xra{\sigma_{F_2}}t_{F_2}\xra{F_2\prec_{\mathcal{G}} F_3}s_{F_3}$$
and so $F_1\prec_{\mathcal{G}} F_3$.
\newline
\emph{Asymmetry}: we have to show that $F_1\prec_{\mathcal{G}} F_2\Rightarrow F_2\cancel{\prec_{\mathcal{G}}}F_1$. Suppose by contradiction that $F_2\prec_{\mathcal{G}} F_1$, then there exists a directed path $t_{F_2}\rightsquigarrow s_{F_1}$ which would give a directed cycle
$$s_{F_1}\xra{\sigma_{F_1}} t_{F_1}\xra{F_1\prec_{\mathcal{G}} F_2} s_{F_2}\xra{\sigma_{F_2}}t_{F_2}\xra{F_2\prec_{\mathcal{G}} F_1} s_{F_1}$$
that contradicts the acyclicity of the pasting scheme.
\end{proof}
\begin{rmk}
If $F_1\prec_{\mathcal{G}} F_2$ then $\sigma_{F_1}$ and $\sigma_{F_2}$ must both lie on a path 
$$s\rightsquigarrow s_{F_1}\xra{\sigma_{F_1}} t_{F_1}\xra{F_1\prec_{\mathcal{G}} F_2} s_{F_2}\xra{\sigma_{F_2}} t_{F_2} \rightsquigarrow t$$
from source to sink of the pasting scheme. 
\end{rmk}
\begin{exmp} The relations $\triangleleft_{\mathcal{G}}$ and $\prec_{\mathcal{G}}$ are unrelated, in the sense that neither one is contained in the other. For instance if $\mathcal{G}$ is the pasting scheme
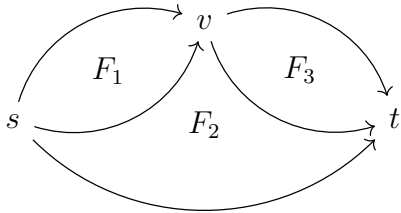
\[\begin{tikzcd}
	&& v \\
	s &&&& t
	\arrow[""{name=0, anchor=center, inner sep=0}, bend left=45, from=2-1, to=1-3]
	\arrow[""{name=1, anchor=center, inner sep=0}, bend left=45, from=1-3, to=2-5]
	\arrow[""{name=2, anchor=center, inner sep=0}, bend right=45, from=2-1, to=1-3]
	\arrow[""{name=3, anchor=center, inner sep=0}, bend right=45, from=1-3, to=2-5]
	\arrow[""{name=4, anchor=center, inner sep=0}, bend right=45, from=2-1, to=2-5]
	\arrow["{F_1}", shorten <=14pt, shorten >=14pt, phantom, from=0, to=2]
	\arrow["{F_2}", shorten <=22pt, shorten >=22pt, phantom, from=1-3, to=4]
	\arrow["{F_3}", shorten <=14pt, shorten >=14pt, phantom, from=1, to=3]
\end{tikzcd}\]
we have $F_1\triangleleft_{\mathcal{G}} F_2$, $F_3\triangleleft_{\mathcal{G}} F_2$ and $F_1\prec_{\mathcal{G}} F_3$.
\end{exmp}
Nevertheless, there exists a connection between the two relations. In fact, the following holds.
\begin{prop}
Let $\triangleleft_{\mathcal{G}}^{t}$ denote the transitive closure of $\triangleleft_{\mathcal{G}}$. We have that $F_1\cancel{\triangleleft_{\mathcal{G}}^{t}}F_2$ and $F_2\cancel{\triangleleft_{\mathcal{G}}^{t}}F_1$ if and only if $F_1\prec_{\mathcal{G}} F_2$ or $F_2\prec_{\mathcal{G}} F_1$.
\end{prop}
\begin{proof}
$\Rightarrow)$ Define $\cG_{F_1}=\{F \ \text{face of} \ \cG\mid F\triangleleft_{\mathcal{G}}^{t}F_1\}$, $\cG_{F_2}=\{F \ \text{face of} \ \cG\mid F\triangleleft_{\mathcal{G}}^{t}F_2\}$ and $\cG_{F_1,F_2}= \cG_{F_1}\cup \cG_{F_2}$. By Remark \ref{properties triangl} we know that $\triangleleft_{\mathcal{G}}$ is both irreflexive and acyclic, hence $\triangleleft_{\mathcal{G}}^{t}$ is still irreflexive. In addition, $F_2\notin \cG_{F_1}$ and $F_1\notin \cG_{F_2}$ by assumption. Therefore, $F_1$ and $F_2$ are not in $\cG_{F_1,F_2}$. Since $\triangleleft_{\mathcal{G}}$ is well founded and $\cG_{F_1,F_2}$ is a subset of the sets of faces of the pasting scheme, there exists a minimal element $H_1$ in $\cG_{F_1,F_2}$. For the same reason, there exists a minimal element $H_2\in \cG_{F_1,F_2}\setminus\{H_1\}$. If we keep removing the $H_i$ for $1\le i\le |\cG_{F_1,F_2}|$ we get a subpasting scheme of $\cG$ that has no faces related to either $F_1$ or $F_2$ in $\triangleleft_{\mathcal{G}}^t$, meaning that $\sigma_{F_1}$ and $\sigma_{F_2}$ are contained in the top path of this subpasting scheme. Therefore $F_1\prec_{\mathcal{G}} F_2$ or $F_2 \prec_{\mathcal{G}} F_1$. 
\newline
$\Leftarrow)$ On the other hand, if $F_1$ and $F_2$ are $\prec_{\mathcal{G}}$-comparable, their domains lie in the same path from source to sink by definition. This is the top path of a subpasting scheme in which $F_1$ and $F_2$ are minimal elements with respect to the relation $\triangleleft_{\mathcal{G}}^t$ restricted to the faces of the subpasting scheme. The elements we removed play no role for $\triangleleft_{\mathcal{G}}^t$ since we cannot find a $H$ sitting between $F_1$ and $F_2$ without contradicting the minimality of one of the two thus $F_1\cancel{\triangleleft_{\mathcal{G}}^{t}}F_2$ and $F_2\cancel{\triangleleft_{\mathcal{G}}^{t}}F_1$. 
\end{proof} 
The relations $\triangleleft_{\mathcal{G}}$ and $\prec_{\mathcal{G}}$ allow us to introduce the categories that we will use in the proof of the pasting theorem. The first one is the following. 
\begin{defn}\label{C_G}
Given a pasting scheme $\cG$, we define $\C_{\cG}$ to be the category with
\begin{enumerate}
    \item \emph{objects:} strings of faces $\overrightarrow{F}=F_1F_2\cdots F_n$ of the pasting scheme corresponding to strict linear extensions of the relation $\triangleleft_{\cG}$; 
    \item \emph{generating morphisms:} for each string of faces $\overrightarrow{F}=F_1F_2\cdots F_n$ and each adjacent pair\footnote{By adjacent we mean that there exists $i\in\{1,\dots,n-1\}$ such that $G=F_i$ and $H=F_{i+1}$.} $G,H$ of faces in $\overrightarrow{F}$ that are not comparable with respect to $\triangleleft_{\cG}^{t}$, a morphism 
    $$\overrightarrow{U}\widehat{GH}\overrightarrow{V}\colon\overrightarrow{F}=\overrightarrow{U}GH\overrightarrow{V}\to\overrightarrow{U}HG\overrightarrow{V}$$
    where $\overrightarrow{U}$ and $\overrightarrow{V}$ are sub-strings of $\overrightarrow{F}$. These morphisms are subject to the relations 
    \begin{enumerate}
        \item $\overrightarrow{U}\widehat{HG}\overrightarrow{V}\circ\overrightarrow{U}\widehat{GH}\overrightarrow{V}=\id$;
        \item $\overrightarrow{U}HG\overrightarrow{V}\widehat{KL}\overrightarrow{W}\circ\overrightarrow{U}\widehat{GH}\overrightarrow{V}KL\overrightarrow{W}=\overrightarrow{U}\widehat{GH}\overrightarrow{V}LK\overrightarrow{W}\circ\overrightarrow{U}GH\overrightarrow{V}\widehat{KL}\overrightarrow{W}$ whenever there is a string of the form $\overrightarrow{U}GH\overrightarrow{V}KL\overrightarrow{W}$ with $\overrightarrow{V}$ possibly equal to the empty string;
        \item $\overrightarrow{U}\widehat{HK}G\overrightarrow{V}\circ\overrightarrow{U}H\widehat{GK}\overrightarrow{V}\circ\overrightarrow{U}\widehat{GH}K\overrightarrow{V}=\overrightarrow{U}K\widehat{GH}\overrightarrow{V}\circ\overrightarrow{U}\widehat{GK}H\overrightarrow{V}\circ\overrightarrow{U}G\widehat{HK}\overrightarrow{V}$ whenever there is a string of the form $\overrightarrow{U}GHK\overrightarrow{V}$. 
    \end{enumerate}
\end{enumerate}
 We will sometimes write $\n\widehat{GH}\n$ instead of $\overrightarrow{U}\widehat{GH}\overrightarrow{V}$ to not overload the notation. The morphisms $\overrightarrow{U}\widehat{GH}\overrightarrow{V}KL\overrightarrow{W}$ and $\overrightarrow{U}\widehat{GH}\overrightarrow{V}LK\overrightarrow{W}$, for example, may be both denoted with $\n\widehat{GH}\n$ so that the placeholder on the right is respectively equal to $\overrightarrow{V}KL\overrightarrow{W}$ or $\overrightarrow{V}LK\overrightarrow{W}$. The same shortcut applies to the other generating morphisms.
\end{defn}
\noindent In $\C_{\cG}$ each generator has an inverse, so every morphism is invertible. In other words $\C_{\cG}$ is a groupoid. However, up to equivalence, it is a very simple groupoid. As a matter of fact, we will prove that the following theorem holds.   
\begin{thrm}
The groupoid $\C_{\cG}$ is contractible, i.e.\ it is equivalent to the terminal category $\bb1$.
\end{thrm}
We will show this result using rewriting (see also \cite{forest2021computational} and \cite{forest2018coherence} for a detailed account of rewriting theory in $\Gray$-categories). In order to be able to apply rewriting techniques we need to choose an orientation for the arrows of $\C_{\cG}$. We will do it by extracting a category $\C_{\cG}'$ which is not a groupoid such that $\C_{\cG}$ is the groupoid reflection of $\C_{\cG}'$, that is the image of $\C_{\cG}'$ under the left adjoint to the inclusion $\Gpd\hookrightarrow\Cat$. In other words, $\C_{\cG}$ is the groupoid obtained by formally inverting every morphism of $\C_{\cG}'$. Using the following lemma, the problem then reduces to proving that $\C_{\cG}'$ has a terminal object.  
\begin{lemma}\label{2-functoriality of grpd refl}
The groupoid reflection of a category $\C$ with a terminal object is contractible.
\end{lemma}
\begin{proof}
The category $\C$ has a terminal object if and only if the unique functor $\C\to\bb1$ has a right adjoint $\bb1\to\C$ picking out the terminal object. The functor sending a category to its groupoid reflection can be promoted to a $2$-functor $\Cat\to\Gpd$ because $\Gpd$ is closed under cotensors over $\Cat$ given that the cotensor of a groupoid $\cD$ by a category $\C$ is the category of functors $[\C,\cD]$ which is itself a groupoid since the components of every natural transformation are invertible. Hence the adjunction between the inclusion and the groupoid reflection can be lifted to a $2$-adjunction by Theorem 4.85 of \cite{kelly1982basic}. Being a $2$-functor, the groupoid reflection sends adjunctions in $\Cat$ to adjunctions in $\Gpd$. But every adjunction in $\Gpd$ is an equivalence because its unit and counit are natural isomorphisms since their components are morphisms inside groupoids (and every morphism in a groupoid is invertible). Therefore the adjunction $\C\rightleftarrows\bb1$ we started with is sent to an equivalence between the groupoid reflection of $\C$ and $\bb1$.
\end{proof}
\begin{defn}\label{C_G'}
Define a category $\C_{\cG}'$ in such a way that $\Ob(\C_{\cG}')=\Ob(\C_{\cG})$ and the morphisms are generated by $\n \widehat{KH}\n$ with $H\prec_{\cG} K$, subject to the relations (b) and (c).
\end{defn}
We interpret these morphisms as rewrite rules on the set $\Ob(\C_{\cG})$. First of all, let us prove that the category we have just defined has $\C_{\cG}$ as its groupoid reflection. 
\begin{lemma}\label{grpd refl}
$\C_{\cG}$ is the groupoid reflection of $\C_{\cG}'$.
\end{lemma}
\begin{proof}
First of all, notice that $\C_{\cG}'[\text{Mor}(\C_{\cG}')^{-1}]=\C_{\cG}'[\text{GenMor}(\C_{\cG}')^{-1}]$, where by $\text{GenMor}(\C_{\cG}')$ we mean the set of generating morphisms. In fact, every morphism in $\C_{\cG}'$ can be written as composition of some generating morphisms and so its inverse is just a composition of the inverses of the generators.
Let us show now that $\C_{\cG}$ has the universal property of the localization with respect to the set of generating morphisms of $\C_{\cG}'$, i.e.\ for every groupoid $\D$ and every functor $F\colon\C_{\cG}'\to\D$ there exists a unique factorization 
\[\begin{tikzcd}
	{\C_{\cG}'} & \D \\
	{\C_{\cG}}
	\arrow["F", from=1-1, to=1-2]
	\arrow["i"', from=1-1, to=2-1]
	\arrow["{\overline{F}}"', dashed, from=2-1, to=1-2]
\end{tikzcd}\]
where $i\colon\C_{\cG}'\to\C_{\cG}$ is the identity on objects and morphisms ($i$ is indeed a functor because every relation in $\C_{\cG}'$ is also a relation in $\C_{\cG}$). We can extend every such functor $F$ to $\overline{F}$ by defining it as $F$ on the objects of $\C_{\cG}$ and as
\[\overline{F}(f) = \begin{cases}
  F(f) & f \text{ is in the image of $i$} \\
  F(f^{-1})^{-1} & f \text{ is not in the image of $i$}
\end{cases}
\]
on the generating morphisms of $\C_{\cG}$, which can be done uniquely since every generating morphism of $\C_{\cG}$ is either a generating morphism of $\C_{\cG}'$ or a inverse to one of the generators. Let us show the compatibility of $\overline{F}$ with the equations of $\C_{\cG}$. We will prove that the relations (b) and (c) in $\C_{\cG}$ can be obtained from the ones in $\C_{\cG}'$ by suitably composing with the inverses of the generating morphisms of $\C_{\cG}'$ (which in turn are generating morphisms in $\C_{\cG}$). Since $\overline{F}$ is defined in terms of $F$, it will then preserve these extra relations. Once we have shown that $\overline{F}$ sends the equations of $\C_{\cG}$ to equalities, we get that $\overline{F}$ can be lifted to another functor having $\C_{\cG}$ as domain that we will still call $\overline{F}$ with an abuse of notation. For the relation (b) we have these cases:
\begin{enumerate}[(i)]
    \item $\n \widehat{GH}\n$ and $\n\widehat{KL}\n$ are morphisms in $\C_{\cG}'$, in which case (b) holds already in $\C_{\cG}'$ hence even more so in $\C_{\cG}$;
    \item neither $\n\widehat{GH}\n$ nor  $\n\widehat{KL}\n$ are morphisms in $\C_{\cG}'$, in which case (b) can be obtained by inverting the relation $\n\widehat{LK}\n\circ\n\widehat{HG}\n=\n\widehat{HG}\n\circ\n\widehat{LK}\n$ that holds in $\C_{\cG}'$ remembering that $\n\widehat{GH}\n=\n\widehat{HG}\n^{-1}$ and $\n\widehat{KL}\n=\n\widehat{LK}\n^{-1}$ 
 in $\C_{\cG}$;
 \item exactly one between $\n\widehat{GH}\n$ and $\n\widehat{KL}\n$ is in $\C_{\cG}'$, e.g.\ $\n\widehat{GH}\n$ (so that $\n\widehat{LK}\n$ is in $\C_{\cG}'$), in which case we have that $H\prec_{\cG}G\prec_{\cG}K\prec_{\cG}L$. We can then get (b) as
 \begin{align*}
     \n HG\n\widehat{KL}\n \circ \n\widehat{GH}\n KL\n &= \n HG\n\widehat{KL}\n \circ \n\widehat{GH}\n KL\n\circ \n GH\n\widehat{LK}\n\circ\n GH\n\widehat{KL}\n \\
     &= \n HG\n\widehat{KL}\n\circ \n HG\n\widehat{LK}\n\circ\n\widehat{GH}\n LK\n\circ \n GH\n\widehat{KL}\n\\
     &= \n\widehat{GH}\n LK\n\circ \n GH\n\widehat{KL}\n
 \end{align*}
  using that $\n\widehat{GH}\n KL\n\circ \n GH\n\widehat{LK}\n=\n HG\n\widehat{LK}\n\circ\n\widehat{GH}\n LK\n$ holds in $\C_{\cG}'$ plus relation (a). Similarly one obtains (b) if $\n\widehat{KL}\n$ is in $\C_{\cG}'$ while $\n\widehat{GH}\n$ is not.
\end{enumerate}
For the relation (c) we have these cases:
\begin{enumerate}[(i)]
    \item $\n\widehat{HK}\n, \n\widehat{GK}\n$ and $\n\widehat{GH}\n$ are all in $\C_{\cG}'$, in which case (c) holds already in $\C_{\cG}'$ hence also in $\C_{\cG}$;
    \item none of the morphisms $\n\widehat{HK}\n, \n\widehat{GK}\n$ and $\n\widehat{GH}\n$ is in $\C_{\cG}'$, in which case we can obtain (c) by inverting the corresponding relation involving the inverses of the morphisms which holds in $\C_{\cG}'$;
    \item two out of three of the morphisms $\n\widehat{HK}\n, \n\widehat{GK}\n$ and $\n\widehat{GH}\n$ are in $\C_{\cG}'$. For composability reasons, the excluded morphism cannot be $\n\widehat{GK}\n$. Suppose that the excluded morphism is $\n\widehat{HK}\n$, so that $\n\widehat{KH}\n$ is in $\C_{\cG}'$. We then have $H\prec_{\cG}K\prec_{\cG}G$ so that we can obtain (c) as
\begin{align*}
    \n\widehat{HK}G\n\circ \n H\widehat{GK}\n\circ\n\widehat{GH}K\n &= \n\widehat{HK}G\n\circ\n H\widehat{GK}\n\circ\n\widehat{GH}K\n\circ \n G\widehat{KH}\n\circ \n G\widehat{HK}\n \\
    &= \n\widehat{HK}G\n\circ\n\widehat{KH}G\n\circ \n K\widehat{GH}\n\circ\n\widehat{GK}H\n\circ \n G\widehat{HK}\n \\
    &= \n K\widehat{GH}\n \circ\n\widehat{GK} H\n\circ\n G\widehat{HK}\n
\end{align*}    
where we used that $\n H\widehat{GK}\n\circ\n\widehat{GH}K\n\circ\n G\widehat{KH}\n=\n\widehat{KH}G\n\circ \n K\widehat{GH}\n\circ\n\widehat{GK}H\n$ in $\C_{\cG}'$ and relation (a). Similarly one obtains (c) if the excluded morphism is $\n\widehat{GH}\n$.
\item one out of three of the morphisms $\n\widehat{HK}\n, \n\widehat{GK}\n$ and $\n\widehat{GH}\n$ is in $\C_{\cG}'$. For composability reasons, one of the two excluded morphisms must be $\n\widehat{GK}\n$. Suppose that the other excluded morphism is $\n\widehat{GH}\n$ so that we have $G\prec_{\cG} K\prec_{\cG} H$ which gives (c) as
\begin{align*}
    \n\widehat{HK}G\n\circ \n H\widehat{GK}\n\circ\n \widehat{GH}K\n &= \n K\widehat{GH}\n\circ\n\widehat{GK}H\n\circ\n\widehat{KG}H\n\circ \n K\widehat{HG}\n\circ\n\widehat{HK}G\n\circ\n H\widehat{GK}\n\circ\n\widehat{GH}K\n \\
    &= \n K\widehat{GH}\n\circ\n\widehat{GK}H\n\circ \n G\widehat{HK}\n\circ\n\widehat{HG} K\n\circ\n H\widehat{KG}\n\circ\n H\widehat{GK}\n\circ\n\widehat{GH}K \n\\
    &= \n K\widehat{GH}\n\circ\n\widehat{GK}H\n\circ \n G\widehat{HK}\n
\end{align*}
where we used that $\n \widehat{KG}H\n\circ\n K\widehat{HG}\n\circ\n\widehat{HK}G \n= \n G\widehat{HK}\n\circ\n\widehat{HG}K\n\circ\n H\widehat{KG}\n$ in $\C_{\cG}'$ and relation (a). Similarly one obtains (c) if the other excluded morphism is $\n\widehat{HK}\n$.
\end{enumerate}
This concludes the proof that $\C_{\cG}$ is the groupoid reflection of $\C_{\cG}'$.
\end{proof}
Therefore we just have to prove that $\C_{\cG}'$ contains a terminal object.
\begin{prop}\label{C_G' term obj}
$\C_{\cG}'$ has a terminal object.
\end{prop}
\begin{proof}
The proof uses Newman's lemma. 
\newline
\emph{Termination:} Let $X_{\overrightarrow{F}}=\{(F_i,F_j) \mid i<j \ \text{and} \ F_j\prec_{\cG} F_i \}$, where $\overrightarrow{F}=F_1\cdots F_n\in\Ob(\C_{\cG}')$. Define the function
\begin{alignat*}{4}
\rho\colon\; && \Ob(\C_{\cG}') && \;\to\;     & \N \\
     && F_1\cdots F_n    && \;\mapsto\; & |X_{\overrightarrow{F}}|
\end{alignat*}
which is reduced by $1$ by any rewrite $\widehat{F_iF_{i+1}}$. In fact, for $\overrightarrow{F'}=F_1\cdots F_{i+1}F_i\cdots F_n$ we have $$X_{\overrightarrow{F}}=X_{\overrightarrow{F'}}\cup\{(F_i,F_{i+1})\}$$ since the only pair of faces that have changed their positions relative to each other is $(F_i,F_{i+1})$. This implies that the rewriting system is terminating because reductions reduce the measure, which is bounded below by $0$.\newline
\emph{Local confluence:} We have to show that every fork (sometimes called \emph{local branching}) can be closed. Recall that if we can apply a rewrite $\widehat{F_iF_j}$ then we know that $F_j\prec_{\cG}F_i$. Given a fork of the kind
\[\begin{tikzcd}
	& {\overrightarrow{F}} \\
	{\overrightarrow{F}'} && {\overrightarrow{F}''}
	\arrow["{\widehat{F_iF_j}}"', from=1-2, to=2-1]
	\arrow["{\widehat{F_hF_k}}", from=1-2, to=2-3]
\end{tikzcd}\]
we consider the sets of faces $\{G,H\}$ and $\{K,L\}$ appearing in its domain, which implies that the elements of each set are $\prec_{\cG}$-comparable. Suppose without loss of generality that $H$ appears before $G$ in the source of the fork, that $L$ appears before $K$, and that $\{G,H\}$ appears before or at the same point as $\{K,L\}$. If $\{G,H\}=\{K,L\}$ there is nothing to prove because there is only one possible swap so no branchings occur. If $\{G,H\}\cap\{K,L\}=\emptyset$ we have a span 
\[\begin{tikzcd}
	& \overrightarrow{U}HG\overrightarrow{V}LK\overrightarrow{W} \\
	\overrightarrow{U}GH\overrightarrow{V}LK\overrightarrow{W} && \overrightarrow{U}HG\overrightarrow{V}KL\overrightarrow{W} 
	\arrow["{\n\widehat{HG}\n}"', from=1-2, to=2-1]
	\arrow["{\n\widehat{LK}\n}", from=1-2, to=2-3]
\end{tikzcd}\]
in the category $\C_{\cG}'$, that can be closed in the following way
\[\begin{tikzcd}
	& \overrightarrow{U}HG\overrightarrow{V}LK\overrightarrow{W} \\
	\overrightarrow{U}GH\overrightarrow{V}LK\overrightarrow{W} && \overrightarrow{U}HG\overrightarrow{V}KL\overrightarrow{W} \\
	& \overrightarrow{U}GH\overrightarrow{V}KL\overrightarrow{W}
	\arrow["{\n\widehat{HG}\n}"', from=1-2, to=2-1]
	\arrow["{\n\widehat{LK}\n}", from=1-2, to=2-3]
	\arrow["{\n\widehat{LK}\n}"', from=2-1, to=3-2]
	\arrow["{\n\widehat{HG}\n}", from=2-3, to=3-2]
\end{tikzcd}\]
because of the relation (b). 
When $\{G,H\}\cap\{K,L\}\ne\emptyset$, we have that $G = L$ and $K\prec_{\cG}G\prec_{\cG}H$ hence there is a span 
\[\begin{tikzcd}
	& \overrightarrow{U}HGK\overrightarrow{V} \\
	\overrightarrow{U}GHK\overrightarrow{V} && \overrightarrow{U}HKG\overrightarrow{V}
	\arrow["{\n\widehat{HG}\n}"', from=1-2, to=2-1]
	\arrow["{\n\widehat{GK}\n}", from=1-2, to=2-3]
\end{tikzcd}\]
in the category $\C_{\cG}'$, that can be closed in the following way
\[\begin{tikzcd}
	& \overrightarrow{U}HGK\overrightarrow{V} \\
	\overrightarrow{U}GHK\overrightarrow{V} && \overrightarrow{U}HKG\overrightarrow{V} \\
	\overrightarrow{U}GKH\overrightarrow{V} && \overrightarrow{U}KHG\overrightarrow{V}  \\
	&  \overrightarrow{U}KGH\overrightarrow{V} 
	\arrow["{\n\widehat{HG}\n}"', from=1-2, to=2-1]
	\arrow["{\n\widehat{GK}\n}", from=1-2, to=2-3]
	\arrow["{\n\widehat{HK}\n}"', from=2-1, to=3-1]
	\arrow["{\n\widehat{GK}\n}"', from=3-1, to=4-2]
	\arrow["{\n\widehat{HK}\n}", from=2-3, to=3-3]
	\arrow["{\n\widehat{HG}\n}", from=3-3, to=4-2]
\end{tikzcd}\]
thanks to the relation (c). Therefore the rewriting system is locally confluent. \newline
Uniqueness of the path from the top to the bottom of the diamond from Newman's lemma comes from the fact that every local branching can be closed \emph{plus} the observation that all the sub-diamonds in the application of Newman's lemma are actually \emph{commutative} diagrams so that at the end all the sub-paths are equal. This is an extra condition which doesn't come from Newman's lemma but from the equations that the generating morphisms in the category $\C_{\cG}'$ have to satisfy. Therefore the minimal element for the rewriting system is a terminal object in $\C_{\cG}'$, which is what we wanted to show.
\end{proof}

\begin{rmk}
It is well known that in a category the terminal element, if it exists, is unique up to a unique isomorphism. In $\C_{\cG}'$ this is even stronger: the terminal element is unique and it coincides with the maximally ordered string of faces, namely the one that doesn't contain adjacent pairs of faces $F, G$ with $F$ coming before $G$ in the string and $G\prec_{\cG}F$. To show uniqueness, suppose there exist two different terminal objects $m$ and $m'$. Then we have a span 
\[\begin{tikzcd}
	& \bullet \\
	m && {m'}
	\arrow[from=1-2, to=2-1]
	\arrow[from=1-2, to=2-3]
\end{tikzcd}\]
that can be closed by Newman's lemma to a commutative diagram 
\[\begin{tikzcd}
	& \bullet \\
	m && {m'} \\
	& n
	\arrow[from=1-2, to=2-1]
	\arrow[from=1-2, to=2-3]
	\arrow[from=2-1, to=3-2]
	\arrow[from=2-3, to=3-2]
\end{tikzcd}\]
because the rewriting system is confluent. But $m$ and $m'$ are minimal, therefore $m=n=m'$.
\end{rmk}
A pasting scheme can be seen as a ``free-living pasting diagram'', meaning that it provides the shape of the pasting diagram which can be then interpreted inside some higher category through a labelling in the following precise sense. \begin{defn}
A \emph{labelling of a pasting scheme} $\cG$ in a $\Gray$-category $\K$ is an assignment of a $0$-cell to each vertex, a $1$-cell to each edge and a $2$-cell to each face of $\cG$ in a way that preserves domains and codomains. 
\end{defn}
In other words,
\begin{enumerate}
    \item we label each vertex $u$ of $\cG$ with an object $l(u)$ of $\cK$,
    \item we label each edge $e$ with a $1$-cell $l(e)$ of $\cK$ so that $\dom_0(l(e))=l(\text{source}(e))$ and $\cod_0(l(e))=l(\text{target}(e))$,
    \item given a directed path $p=e_1e_2\dots e_n$ we define $l(p)\coloneqq l(e_n)\circ l(e_{n-1})\circ\dots l(e_2)\circ l(e_1)$ which is uniquely defined because horizontal composition of $1$-cells is associative,
    \item we label every face $F$ with a $2$-cell $l(F)$ such that $\dom_1(l(F))=l(\sigma_F)$ and $\cod_1(l(F))=l(\tau_F)$.
\end{enumerate}
Sometimes we call \emph{pasting diagram} the image of this assignment. To be precise, this is what we refer to as \emph{$2$-dimensional pasting diagram} in the main theorem. The bidimensionality comes from the planarity of the pasting scheme.

The relations $\triangleleft_{\cG}$ and $\prec_{\mathcal{G}}$ are a way to capture two different features of the process that returns a composite of a given pasting diagram, which is a vertical composition of whiskered $2$-cells. The relation $\triangleleft_{\mathcal{G}}$ specifies the order in which the $2$-cells appear in this vertical composition: if $F_1\triangleleft_{\mathcal{G}}F_2$ then any composite of the corresponding pasting diagram will have a term of the form $l(F_2)\cdot l(F_1)$ - or a suitable whiskering of it - in this precise order. The relation $\prec_{\mathcal{G}}$ has instead to do with the application of the middle four interchange law to a pair of $2$-cells whose preimages via the labelling are $\prec_{\cG}$-compatible.

Given a pasting scheme $\cG$, the objects of the category $\C_{\cG}$ and the composites of the corresponding pasting diagram are related in a way such that the axioms chosen for $\C_{\cG}$ are compatible with the ones of $\Gray$-categories. This is the content of the following lemma.
\begin{lemma}\label{compatibility} 
Given a labelling into a $\Gray$-category $\cK$ of a pasting scheme $\cG$, there is a canonical functor from $\C_{\cG}$ to $\K(S,T)(p,q)$ where $S, T, p, q$ are the images of the source, sink, top and bottom path of $\cG$ respectively. 
\end{lemma}
\begin{proof} 
 The functor $\C_{\cG}\to\K(S,T)(p,q)$ is defined by sending a string of faces $\overrightarrow{F}=F_1\cdots F_n$ to the $2$-cell $\alpha$ obtained by composing the $\alpha_i=l(F_i)$ in the same order\footnote{Notice that composing $\alpha_i$ before $\alpha_j$ for $i<j$ means that they appear in the \emph{reverse} order in the expression of $\alpha$ as a vertical composite of whiskerings of these $2$-cells.}. If we have a pasting scheme of the form   
 \[\begin{tikzcd}
	{s} & {v} & {t}
	\arrow[""{name=0, anchor=center, inner sep=0}, "{e_1'}"', bend right=45, from=1-1, to=1-2]
	\arrow[""{name=1, anchor=center, inner sep=0}, "{e_2}", bend left=45, from=1-2, to=1-3]
	\arrow[""{name=2, anchor=center, inner sep=0}, "{e_1}", bend left=45, from=1-1, to=1-2]
	\arrow[""{name=3, anchor=center, inner sep=0}, "{e'_2}"', bend right=45, from=1-2, to=1-3]
	\arrow["F"{description}, draw=none, from=2, to=0]
	\arrow["G"{description}, draw=none, from=1, to=3]
\end{tikzcd}\]
 labelled as follows (notice that in this case $p=fg$ and $q=f'g'$)
 \[\begin{tikzcd}
	S & A & T
	\arrow[""{name=0, anchor=center, inner sep=0}, "{g'}"', bend right=45, from=1-1, to=1-2]
	\arrow[""{name=1, anchor=center, inner sep=0}, "f", bend left=45, from=1-2, to=1-3]
	\arrow[""{name=2, anchor=center, inner sep=0}, "g", bend left=45, from=1-1, to=1-2]
	\arrow[""{name=3, anchor=center, inner sep=0}, "{f'}"',  bend right=45, from=1-2, to=1-3]
	\arrow["\psi", shorten <=6pt, shorten >=6pt, Rightarrow, from=2, to=0]
	\arrow["\varphi", shorten <=6pt, shorten >=6pt, Rightarrow, from=1, to=3]
\end{tikzcd}\]
 we send the generating morphism $\widehat{FG}$ to the image of the Gray cell
 \[\begin{tikzcd}
	{(f,g)} & {(f,g')} \\
	{(f',g)} & {(f',g')}
	\arrow["{(f,\psi)}", from=1-1, to=1-2]
	\arrow["{(\varphi, g')}", from=1-2, to=2-2]
	\arrow["{(\varphi, g)}"', from=1-1, to=2-1]
	\arrow["{(f',\psi)}"', from=2-1, to=2-2]
	\arrow["{\gamma_{\varphi,\psi}}"', shorten <=12pt, shorten >=12pt, Rightarrow, from=1-2, to=2-1]
\end{tikzcd}\]
through the composition $2$-functor $c_{S,A,T}\colon\cK(A,T)\otimes\cK(S,A)\to\cK(S,T)$. For a general pasting scheme we send a generating morphism $$\overrightarrow{U}\widehat{GH}\overrightarrow{V}\colon\overrightarrow{F}=\overrightarrow{U}GH\overrightarrow{V}\to\overrightarrow{U}HG\overrightarrow{V}$$ to a suitable whiskering of the Gray cell $\gamma_{l(H),l(G)}$ with the other $2$-cells that appear in the labelling of the pasting scheme. 
It remains to show that this assignment preserves the relations that hold in $\C_{\cG}$, i.e.\ that it is functorial. Relation (a) is preserved because the Gray cells are invertible. Relation (b) is preserved because for a pasting scheme labelled as follows
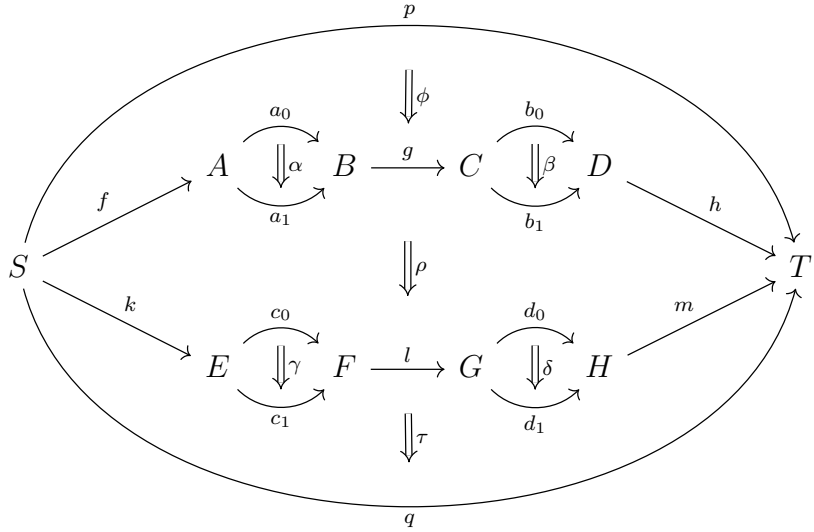
\[\begin{tikzcd}
	&& A & B & C & D \\
	S &&&&&&& T \\
	&& E & F & G & H
	\arrow["f", from=2-1, to=1-3]
	\arrow["k", from=2-1, to=3-3]
	\arrow[""{name=0, anchor=center, inner sep=0}, "a_0", bend left=45, from=1-3, to=1-4]
	\arrow[""{name=1, anchor=center, inner sep=0}, "a_1"', bend right=45, from=1-3, to=1-4]
	\arrow[""{name=2, anchor=center, inner sep=0}, "g", from=1-4, to=1-5]
	\arrow[""{name=3, anchor=center, inner sep=0}, "b_0", bend left=45, from=1-5, to=1-6]
	\arrow[""{name=4, anchor=center, inner sep=0}, "b_1"', bend right=45, from=1-5, to=1-6]
	\arrow["h", from=1-6, to=2-8]
	\arrow[""{name=5, anchor=center, inner sep=0}, "c_0", bend left=45, from=3-3, to=3-4]
	\arrow[""{name=6, anchor=center, inner sep=0}, "c_1"', bend right=45, from=3-3, to=3-4]
	\arrow[""{name=7, anchor=center, inner sep=0}, "l", from=3-4, to=3-5]
	\arrow[""{name=8, anchor=center, inner sep=0}, "d_0", bend left=45, from=3-5, to=3-6]
	\arrow[""{name=9, anchor=center, inner sep=0}, "d_1"', bend right=45, from=3-5, to=3-6]
	\arrow["m", from=3-6, to=2-8]
	\arrow[""{name=10, anchor=center, inner sep=0}, "p", bend left=75, from=2-1, to=2-8]
	\arrow[""{name=11, anchor=center, inner sep=0}, "q"', bend right=75, from=2-1, to=2-8]
	\arrow["\alpha", shorten <=6pt, shorten >=6pt, Rightarrow, from=0, to=1]
	\arrow["\beta", shorten <=6pt, shorten >=6pt, Rightarrow, from=3, to=4]
	\arrow["\gamma", shorten <=6pt, shorten >=6pt, Rightarrow, from=5, to=6]
	\arrow["\delta", shorten <=6pt, shorten >=6pt, Rightarrow, from=8, to=9]
	\arrow["\rho", shorten <=27pt, shorten >=27pt, Rightarrow, from=2, to=7]
	\arrow["\phi", shorten <=16pt, shorten >=16pt, Rightarrow, from=10, to=2]
	\arrow["\tau", shorten <=16pt, shorten >=16pt, Rightarrow, from=7, to=11]
\end{tikzcd}\]
there is a corresponding horizontal composite 
\[\begin{tikzcd}
	&& hb_0ga_1f &&& md_0lc_1k \\
	p & hb_0ga_0f && hb_1ga_1f & md_0lc_0k && md_1lc_1k & q \\
	&& hb_1ga_0f &&& md_1lc_0k
	\arrow["\phi", from=2-1, to=2-2]
	\arrow["{hb_0g\alpha f}", sloped, from=2-2, to=1-3]
	\arrow["{h\beta g a_1 f}", sloped, from=1-3, to=2-4]
	\arrow["{h\beta ga_0 f}"', sloped, from=2-2, to=3-3]
	\arrow["{h b_1g\alpha f}"', sloped, from=3-3, to=2-4]
	\arrow["\rho", from=2-4, to=2-5]
	\arrow["{md_0l\gamma k}", sloped, from=2-5, to=1-6]
	\arrow["{m\delta l c_0 k}"', sloped, from=2-5, to=3-6]
	\arrow["{m\delta l c_1 k}", sloped, from=1-6, to=2-7]
	\arrow["{md_1l\gamma k}"', sloped, from=3-6, to=2-7]
	\arrow["\tau", from=2-7, to=2-8]
	\arrow["hc_{S,B,T}(\gamma_{\beta g, \alpha})f"description, shorten <=16pt, shorten >=16pt, Rightarrow, from=1-3, to=3-3]
	\arrow["mc_{S,F,T}(\gamma_{\delta l, \gamma})k"description, shorten <=16pt, shorten >=16pt, Rightarrow, from=1-6, to=3-6]
\end{tikzcd}\]
of $2$-cells in the hom-$2$-category $\cK(S,T)$, in which the middle four interchange law holds strictly. Relation (c) is preserved because to a pasting scheme labelled as follows 
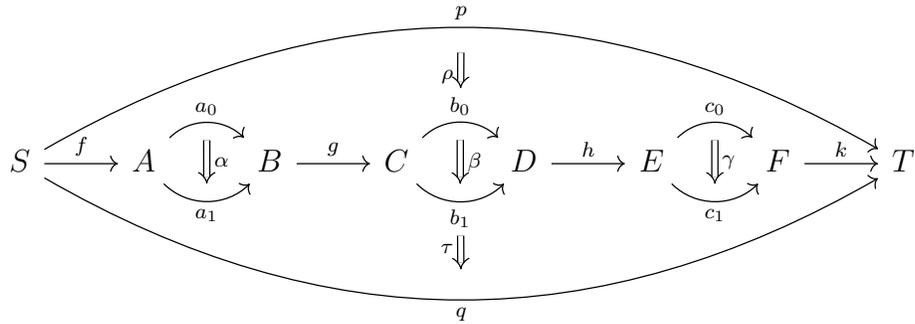
\[\begin{tikzcd}
	S & A & B & C & D & E & F & T
	\arrow["f", from=1-1, to=1-2]
	\arrow[""{name=0, anchor=center, inner sep=0}, "a_0", bend left=45, from=1-2, to=1-3]
	\arrow[""{name=1, anchor=center, inner sep=0}, "a_1"', bend right=45, from=1-2, to=1-3]
	\arrow["g", from=1-3, to=1-4]
	\arrow[""{name=2, anchor=center, inner sep=0}, "b_0", bend left=45, from=1-4, to=1-5]
	\arrow[""{name=3, anchor=center, inner sep=0}, "b_1"', bend right=45, from=1-4, to=1-5]
	\arrow["h", from=1-5, to=1-6]
	\arrow[""{name=4, anchor=center, inner sep=0},"c_0", bend left=45, from=1-6, to=1-7]
	\arrow[""{name=5, anchor=center, inner sep=0}, "c_1"', bend right=45, from=1-6, to=1-7]
	\arrow["k", from=1-7, to=1-8]
	\arrow[""{name=6, anchor=center, inner sep=0}, "p", bend left=30, from=1-1, to=1-8]
	\arrow[""{name=7, anchor=center, inner sep=0}, "q"', bend right=30, from=1-1, to=1-8]
	\arrow["\rho"', shorten <=23pt, shorten >=28pt, yshift=5mm, Rightarrow, from=6, to=3]
	\arrow["\tau"', shorten <=28pt, shorten >=25pt, yshift=-5mm, Rightarrow, from=2, to=7]
	\arrow["\alpha", shorten <=6pt, shorten >=6pt, Rightarrow, from=0, to=1]
	\arrow["\beta", shorten <=6pt, shorten >=6pt, Rightarrow, from=2, to=3]
	\arrow["\gamma", shorten <=6pt, shorten >=6pt, Rightarrow, from=4, to=5]
\end{tikzcd}\]
corresponds a whiskered version of the cube identity 
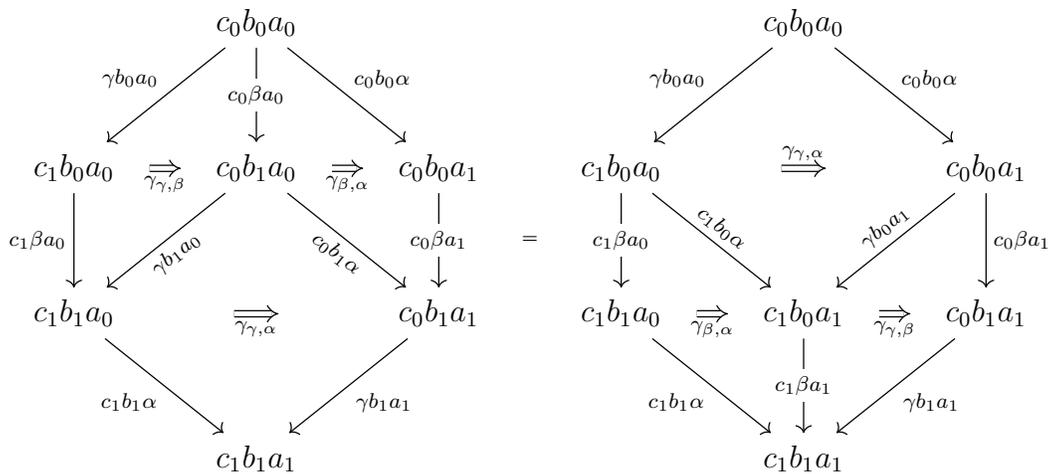
\[\begin{tikzcd}[row sep=.5in]
	& c_0b_0a_0 &&& c_0b_0a_0 \\
	c_1b_0a_0 & c_0b_1a_0 & c_0b_0a_1 & c_1b_0a_0 && c_0b_0a_1 \\
	c_1b_1a_0 && c_0b_1a_1 & c_1b_1a_0 & c_1b_0a_1 & c_0b_1a_1 \\
	& c_1b_1a_1 &&& c_1b_1a_1
	\arrow["{\gamma b_0a_0}"', from=1-2, to=2-1]
	\arrow["{c_1\beta a_0}"', from=2-1, to=3-1]
	\arrow["{c_1b_1\alpha}"', from=3-1, to=4-2]
	\arrow["{c_0\beta a_0}"{description}, from=1-2, to=2-2]
	\arrow["{\gamma b_1 a_0}"', sloped, from=2-2, to=3-1]
	\arrow["{c_0b_0\alpha}", from=1-2, to=2-3]
	\arrow[""{name=0, anchor=center, inner sep=0}, "{c_0\beta a_1}"{description}, from=2-3, to=3-3]
	\arrow["{c_0b_1\alpha}"', sloped, from=2-2, to=3-3]
	\arrow["{\gamma b_1 a_1}", from=3-3, to=4-2]
	\arrow["{\gamma b_0a_0}"', from=1-5, to=2-4]
	\arrow[""{name=1, anchor=center, inner sep=0}, "{c_1\beta a_0}"{description}, from=2-4, to=3-4]
	\arrow["{c_1b_1\alpha}"', from=3-4, to=4-5]
	\arrow["{c_1b_0\alpha}", sloped, from=2-4, to=3-5]
	\arrow["{c_1\beta a_1}"{description}, from=3-5, to=4-5]
	\arrow["{c_0b_0\alpha}", from=1-5, to=2-6]
	\arrow["{\gamma b_0a_1}", sloped, from=2-6, to=3-5]
	\arrow["{c_0\beta a_1}", from=2-6, to=3-6]
	\arrow["{\gamma b_1 a_1}", from=3-6, to=4-5]
	\arrow["{\gamma_{\gamma,\beta}}"', shorten <=8pt, shorten >=8pt, Rightarrow, from=2-1, to=2-2]
	\arrow["{\gamma_{\beta,\alpha}}"', shorten <=8pt, shorten >=8pt, Rightarrow, from=2-2, to=2-3]
	\arrow["{\gamma_{\gamma, \alpha}}"', shorten <=40pt, shorten >=40pt, Rightarrow, from=3-1, to=3-3]
	\arrow["{\gamma_{\gamma,\alpha}}", shorten <=40pt, shorten >=40pt, Rightarrow, from=2-4, to=2-6]
	\arrow["{\gamma_{\beta,\alpha}}"', shorten <=8pt, shorten >=8pt, Rightarrow, from=3-4, to=3-5]
	\arrow["{\gamma_{\gamma,\beta}}"', shorten <=8pt, shorten >=8pt, Rightarrow, from=3-5, to=3-6]
	\arrow["{=}"{description}, Rightarrow, draw=none, from=0, to=1]
\end{tikzcd}\]
for the Gray tensor product, which is a consequence of naturality applied to Gray cells\footnote{In the cube identity we omitted the intermediate $1$-cells and denoted again with $\gamma_{\bullet,\bullet}$ the image of a Gray cell under the appropriate composition $2$-functor to not overload the notation.}. 
\end{proof}
We move now to the proof of the main theorem.
\begin{thrm}\label{pasting theorem}
Every $2$-dimensional pasting diagram in a $\Gray$-category has a unique composition up to a contractible groupoid of choices.
\end{thrm}
\begin{proof}
Given a $2$-dimensional pasting diagram in a $\Gray$-category $\cK$, image of the pasting scheme $\cG$ via a labelling, we can encode its composites in the groupoid $\C_{\cG}$ of Definition \ref{C_G} thanks to Lemma \ref{compatibility}. From Lemma \ref{grpd refl} we know that the latter is the groupoid reflection of the category $\C_{\cG}'$ introduced in Definition \ref{C_G'}, that has a terminal object by Proposition \ref{C_G' term obj}. Therefore by Lemma \ref{2-functoriality of grpd refl} the groupoid $\C_{\cG}$ is equivalent to the terminal category $\bb1$, namely it is contractible. This concludes the proof. 
\end{proof}
\begin{rmk}
This proof can be also interpreted inside a lax $\Gray$-category. In this case the composition of a pasting diagram is no longer unique up to a contractible groupoid of choices but there is still a ``minimal choice'' for it, namely the one corresponding to the image of the terminal object of the category $\C_{\cG}'$ via a labelling. In fact, in defining $\C_{\cG}'$ we chose an orientation for the rewrites and this can be seen in turn as choosing a direction to the Gray cell expressing the middle four interchange isomorphism for our $\Gray$-category. 
\end{rmk}

\printbibliography

@phdthesis{johnson1987pasting,
  title={Pasting diagrams in $n$-categories with applications to coherence theorems and categories of paths},
  author={Johnson, Michael},
  year={1987},
  school={University of Sydney}
}

@book{baader1999term,
  title={Term rewriting and all that},
  author={Baader, Franz and Nipkow, Tobias},
  year={1999},
  publisher={Cambridge university press}
}

@article{huet1980confluent,
  title={Confluent reductions: Abstract properties and applications to term rewriting systems},
  author={Huet, G{\'e}rard},
  journal={Journal of the ACM (JACM)},
  volume={27},
  number={4},
  pages={797--821},
  year={1980},
  publisher={ACM New York, NY, USA}
}

@book{barendsen2003term,
  title={Term rewriting systems},
  author={Terese},
  year={2003},
  publisher={Cambridge Tracts in Theoretical Computer Science 55. Cambridge University Press}
}

@inproceedings{forest2018coherence,
  title={Coherence of Gray categories via rewriting},
  author={Forest, Simon and Mimram, Samuel},
  booktitle={3rd International Conference on Formal Structures for Computation and Deduction (FSCD 2018)},
  volume={108},
  year={2018},
  organization={Schloss Dagstuhl--Leibniz-Zentrum fuer Informatik}
}

@misc{NDVmres,
  author       = {Di Vittorio, Nicola}, 
  title        = {2-derivators},
  note         = {Master of Research Thesis, Macquarie University, },
  howpublished = {\url{https://doi.org/10.25949/19817653.v1}},  
  year = {2020}
  }

@book{gordon1995coherence,
  title={Coherence for tricategories},
  author={Gordon, Robert and Power, Anthony John and Street, Ross},
  volume={558},
  year={1995},
  publisher={American Mathematical Soc.}
}

@article{power19902,
  title={A 2-categorical pasting theorem},
  author={Power, A. John},
  journal={Journal of Algebra},
  volume={129},
  number={2},
  pages={439--445},
  year={1990},
  publisher={Elsevier}
}

@article{verity2011enriched,
  title={Enriched categories, internal categories and change of base},
  author={Verity, Dominic},
  journal={Reprints in {T}heory and {A}pplications of {C}ategories},
  number={20},
  pages={1--266},
  year={2011},
  publisher={Mount Allison University, Department of Mathematics and Science}
}

@book{kelly1982basic,
  title={Basic concepts of enriched category theory},
  author={Kelly, G. M.},
  volume={10},
  year={2005},
  publisher={Reprints in Theory and Applications of Categories}
}

@book{johnson20212,
  title={2-dimensional categories},
  author={Johnson, Niles and Yau, Donald},
  year={2021},
  publisher={Oxford University Press, USA}
}

@phdthesis{forest2021computational,
  title={Computational descriptions of higher categories},
  author={Forest, Simon},
  year={2021},
  school={Institut Polytechnique de Paris}
}

@mish{hackney2021infty,
  title={An $(\infty, 2)$-categorical pasting theorem},
  author={Hackney, Philip and Ozornova, Viktoriya and Riehl, Emily and Rovelli, Martina},
  howpublished={arXiv:2106.03660,},
  year={2021}
}

@book{RVbook,
  title={Elements of $\infty$-Category Theory},
  author={Riehl, Emily and Verity, Dominic},
  volume={194},
  year={2022},
  publisher={Cambridge University Press}
}

\end{document}